\documentclass[11pt,reqno,a4paper]{amsproc}
\usepackage[margin=1in]{geometry}
\usepackage{amsmath, amsthm, amssymb}
\usepackage[colorlinks=true, pdfstartview=FitV, linkcolor=blue,citecolor=blue, urlcolor=blue]{hyperref}
\usepackage[abbrev,lite,nobysame]{amsrefs}
\usepackage{times}
\usepackage[usenames,dvipsnames]{color}

\usepackage{mathtools}

\mathtoolsset{showonlyrefs=true}

\usepackage{dsfont}

\def\eps{\varepsilon}
\def\e{{\rm e}} 
\def\dd{{\rm d}}
\def\ddt{{\frac{\dd}{\dd t}}}
\def\R {\mathbb{R}}
\def\Z {\mathbb{Z}}
\def\N {\mathbb{N}}
\def\J {\mathcal{J}}

\def \l {\langle}
\def \r {\rangle}
\def\T {{\mathbb T}}
\def\UU {{\mathfrak U}}
\def\de{{\partial}}

\def\aJ {{\alpha}}
\def\bJ {{\beta}}
\def\cJ {{\gamma}}
\def\ff{{\widehat f}}

\newtheorem{proposition}{Proposition}[section]
\newtheorem{theorem}[proposition]{Theorem}

\newtheorem{lemma}[proposition]{Lemma}
\theoremstyle{definition}

\newtheorem{remark}[proposition]{Remark}

\numberwithin{equation}{section}

\title[Stable mixing estimates in the infinite P\'eclet number limit]{Stable mixing estimates in the infinite P\'eclet number limit}

\author[M. Coti Zelati]{Michele Coti Zelati$^\sharp$}
\thanks{$^\sharp$\textit{m.coti-zelati@imperial.ac.uk}, Department of Mathematics, Imperial College London, London, SW7 2AZ, UK}


\begin{document}

\begin{abstract}
We consider a passive scalar $f$ advected by a strictly monotone shear flow and with a diffusivity parameter $\nu\ll 1$.
We prove an estimate on the homogeneous $\dot{H}^{-1}$ norm of $f$ that combines both the $L^2$ enhanced diffusion 
effect at a sharp rate proportional to $\nu^{1/3}$, and the sharp mixing decay proportional to $t^{-1}$ of the $\dot{H}^{-1}$ norm 
of $f$ when $\nu=0$.
In particular, the estimate is stable in the infinite P\'eclet number limit, as $\nu\to 0$. To the best of our knowledge, this is the first result
of this kind since the work of Kelvin in 1887 on the Couette flow.

The two key ingredients in the proof are an adaptation of the hypocoercivity method and the use of a vector field $J$ that commutes with 
the transport part of the equation. The $L^2$ norm of $Jf$ together with the $L^2$ norm of $f$ produces a suitable upper bound for the  
$\dot{H}^{-1}$ norm of the solution that gives the extra decay factor of $t^{-1}$.
\end{abstract}


\maketitle


\section{Introduction}
Questions related to the understanding of the interaction between mixing and diffusion in fluid mechanics date back
to the fundamental works of Kelvin \cite{Kelvin87} and Reynolds \cite{Reynolds83}. When considering a passive scalar 
$f:[0,\infty)\times \T\times \R\to \R$, advected by a shear $u:\R\to\R$ and diffused, one would like to study in a quantitative way the long-time behavior
of solutions to the drift-diffusion equation
\begin{align}\label{eq:DD}
\begin{cases}
\de_t f+u\de_xf=\nu \Delta f, &\quad \text{in } (0,\infty)\times \T\times \R,\\
f(0,x,y)=f^{in}(x,y),  &\quad\text{in } \T\times \R.
\end{cases}
\end{align}
Here, $\nu>0$ is the diffusivity coefficient, proportional to the inverse P\'eclet number, and $f^{in}$ is a given mean-free initial datum.
The case of the Couette flow $u(y)=y$ is an enlightening example: the equation can be solved explicitly using the 
Fourier transform, giving
\begin{align}\label{eq:couettesol}
\ff(t,\ell,\eta)=\widehat{f^{in}}(\ell,\eta+\ell t) \exp\left\{-\nu\int_0^t \left[\ell^2+|\eta+\ell t- \ell\tau|^2\right]\dd \tau\right\}, \qquad (\ell,\eta)\in \Z\times \R.
\end{align}
The term $\widehat{f^{in}}(\ell,\eta+\ell t)$ is due to \emph{inviscid mixing}. For $\ell\neq 0$, it implies a transfer of information to high frequencies,
producing, even when $\nu=0$, a decay of solutions proportional to $1/t$ in the negative $H^{-1}_{y}$ Sobolev norm, namely
\begin{align}\label{eq:mixcouetteFour}
\|\ff(t,\ell)\|_{H^{-1}_y}:=\int_{\R} \frac{|\ff(t,\ell,\eta)|^2}{\ell^2+\eta^2}\dd \eta\leq \frac{2}{\sqrt{1+(\ell t)^2}} \|\widehat{f^{in}}(\ell)\|_{H^{1}_y},\qquad \forall t\geq0.
\end{align}
Consequently, we find
\begin{align}\label{eq:mixcouette}
 \qquad\|f_{\neq}(t)\|_{\dot{H}^{-1}}\leq \frac{2}{\sqrt{1+t^2}} \|f^{in}_{\neq}\|_{\dot{H}^{1}},\qquad \forall t\geq0,
\end{align}
where $f_{\neq}$ denotes the projection of $f$ onto non-zero modes in $x$, i.e. $f$ minus its $x$-average.
Note that the $x$-average of the equation is a conserved quantity for \eqref{eq:DD} when $\nu=0$, so there is no hope 
of decay for the $\ell=0$ mode.
On the other hand,
for $\nu>0$, the (super) exponential factor in \eqref{eq:couettesol} causes a decay in $L^2$ of the type
\begin{align}\label{eq:encouetteFour}
\|\ff(t,\ell)\|_{L^2_\eta}\leq  \e^{-\frac{1}{12} \nu \ell^2t^3}\|\widehat{f^{in}}(\ell)\|_{L^2_\eta},\qquad \forall t\geq0.
\end{align}
In particular, summing over all $\ell\neq 0$, we obtain
\begin{align}\label{eq:encouette}
\|f_{\neq}(t)\|_{L^2}\leq \e^{-\frac{1}{12} \nu t^3} \|f^{in}_{\neq}\|_{L^2},\qquad \forall t\geq0.
\end{align}
We refer to this estimate as \emph{enhanced diffusion}, as it implies homogenization at a time-scale $O(\nu^{-\frac13})$, hence
much shorter than the diffusive time-scale $O(\nu^{-1})$. Again, the $\ell=0$ mode satisfies the one-dimensional heat equation, so 
the restrictions to all nonzero modes in $x$ is necessary. The intriguing fact about the Couette flow is that the two estimates
\eqref{eq:mixcouette} and \eqref{eq:encouette} can be combined in a single one, in the sense that one can prove that
\begin{align}\label{eq:enMIXcouette}
\|f_{\neq}(t)\|_{\dot{H}^{-1}}\leq \frac{2\e^{-\frac{1}{12} \nu t^3}}{\sqrt{1+t^2}} \|f^{in}_{\neq}\|_{\dot{H}^1},\qquad \forall t\geq0.
\end{align}
The above \eqref{eq:enMIXcouette} is what we refer to as a \emph{stable mixing estimate} in the infinite P\'eclet number limit, since
there is no harm in setting $\nu=0$ and still obtaining non-trivial information about $f$, unlike in \eqref{eq:encouette}. Needless to say, 
having the explicit solution \eqref{eq:couettesol} at our disposal is crucial in deducing such estimate, and in fact this is essentially what Kelvin 
did in \cite{Kelvin87} back in 1887. To this date, \eqref{eq:enMIXcouette} remains the only stable mixing estimate available in the literature.

\subsection{The main results}
The goal of this paper is to prove stable mixing estimates for \eqref{eq:DD} for a large class of strictly monotone shear flows. Assume
that $u\in C^3$ and that there exists $\UU\geq1$ such that
\begin{align}\tag{H}\label{eq:shear}
\frac{1}{\UU} \leq u'(y), \qquad \frac{|u''(y)|}{u'(y)}\leq \UU,\ \qquad \frac{|u'''(y)|}{u'(y)}\leq \UU,\qquad \forall y\in \R.
\end{align}
Let $f$ be the solution to \eqref{eq:DD} or to the hypoelliptic drift-diffusion equation 
\begin{align}\label{eq:ADE}
\begin{cases}
\de_t f+u\de_xf=\nu \de_{yy} f, &\quad \text{in } (0,\infty)\times \T\times \R,\\
f(0,x,y)=f^{in}(x,y),  &\quad\text{in } \T\times \R.
\end{cases}
\end{align}
Consider the $L^2$-weighted norm defined by
\begin{equation}\label{eq:Xnorm}
	\|f\|_{u'}^2:=\|f\|_{L^2}^2+\|u'f\|^2_{L^2}.
\end{equation}
Throughout the article, the initial datum $f^{in}$ is required to satisfy $\|f^{in}_{\neq}\|_{u'}<\infty$.
Our main result is a detailed study of the longtime dynamics of  \eqref{eq:ADE}, including a stable mixing estimate resembling \eqref{eq:enMIXcouette}.
\begin{theorem}\label{thm:main}
There exist $C_0\geq 1$ and $\eps_0,\nu_0\in(0,1)$ only depending on $\UU$ such that if $\nu\in[0, \nu_0]$ we have the enhanced diffusion estimate
\begin{align}\label{eq:engen}
\|f_{\neq}(t)\|_{u'}\leq C_0 \e^{-\eps_0 \nu^{1/3} t}\|f^{in}_{\neq}\|_{u'},\qquad \forall t\geq 0,
\end{align}
and the stable mixing estimate
\begin{align}\label{eq:enMIXmono}
\|f_{\neq}(t)\|_{\dot{H}^{-1}}\leq C_0\frac{ \e^{-\eps_0 \nu^{1/3} t}}{\sqrt{1+t^2}}\left[\|f^{in}_{\neq}\|_{u'}+\|\de_y f^{in}_{\neq}\|_{u'}\right],\qquad  \forall t\geq 0.
\end{align}
In particular, if $\nu=0$ we obtain the inviscid mixing estimate
\begin{align}\label{eq:mixgen}
\|f_{\neq}(t)\|_{\dot{H}^{-1}}\leq \frac{C_0}{\sqrt{1+t^2}}\left[\|f^{in}_{\neq}\|_{u'}+\|\de_y f^{in}_{\neq}\|_{u'}\right],\qquad  \forall t\geq 0.
\end{align}
All the constants can be computed explicitly.
\end{theorem}
Before we proceed to explain the main ideas of the proofs and expand on the concepts of mixing and enhanced diffusion, a few remarks
are in order.

\begin{remark}[On the class of shear flows]
In addition to strict monotonicity, assumption \eqref{eq:shear} essentially imposes a growth condition at $|y|\to\infty$ on $u$ and its derivatives.
Besides small perturbations of the Couette flow, such as
\begin{align}
u(y)=y+\frac12\sin y,
\end{align}
it allows up to exponential growth at infinity, including
\begin{align}
u(y)=y+\e^y, \qquad \text{and}\qquad  u(y)=y(1+|y|^{n-1}),
\end{align}
for $n\geq 3$ (for $n=2$, one has to smooth out at $y=0$, but of course quadratic grow at infinity is allowed). What is not included
is highly oscillatory shear flows, such as
\begin{align}
u(y)=y+\frac12\int_0^y\sin (z^2)\dd z,
\end{align}
for which $\frac12\leq u'\leq \frac32$, but such that $u''(y)=y\cos (y^2)$, violating  \eqref{eq:shear}.
\end{remark}

\begin{remark}[On the weighted norm]
The choice of the norm \eqref{eq:Xnorm} appears naturally in the proof, and precisely accounts for the possible growth at infinity of $u'$.
In the particular case $u(y)=y^2$, this choice has been made in \cite{CZEW19} for proving an enhanced dissipation estimate of the type
\eqref{eq:engen} for the vorticity in the Navier-Stokes equations near the Poiseuille flow. Usually, this choice implies a logarithmic loss in the decay rate \eqref{eq:engen}, while in our case, due to strict monotonicity, this loss can be avoided.  Notice that if $u'$ is assumed to be bounded, the
norm simply reduces to the standard $L^2$ norm.
\end{remark}

\begin{remark}[On the super exponential decay]
The enhanced diffusion estimate \eqref{eq:engen} is a semigroup estimate for the linear operator appearing in \eqref{eq:ADE}.
Whether a super-exponential decay of the type present in \eqref{eq:encouette} or \eqref{eq:enMIXcouette} holds for flows different than the
Couette flow is an open question, and it is related to the spectral properties of the linear operator $L=iku(y)-\nu\de_{yy}$. In particular,
when $u(y)=y$, the operator $L$ has empty spectrum, which is a necessary condition for a semigroup to decay faster than exponentially.
\end{remark}

\begin{remark}[Hypoellipticity]
All the estimates of Theorem \ref{thm:main} can be made more precise by stating them for each $x$-Fourier mode, similarly 
to \eqref{eq:mixcouetteFour} and \eqref{eq:encouetteFour}. In particular, it can be proven that for each $\ell\neq 0$ there holds
\begin{align}\label{eq:engenFour}
\|\ff(t,\ell)\|_{u'}\leq C_0 \e^{-\eps_0 \nu^{1/3}|\ell|^{2/3} t}\|\widehat{f^{in}}(\ell)\|_{u'},
\end{align}
for every $t\geq 0$. As noted in \cite{BCZ15}, this estimate is particularly relevant for the hypoelliptic equation \eqref{eq:ADE}, since
it provides a quantification of the instantaneous regularization from $L^2$ to Gevrey-$\frac32$ even if no diffusion in $x$ is present 
in the equation. For a precise statement, see Theorem \ref{thm:mainFour} below.
\end{remark}

\begin{remark}[The Batchelor scale]
The characteristic filamentation length
scale of $f$ can be defined in terms of the ratio
\begin{align}
\lambda(t)=\frac{\|f(t)\|_{\dot{H}^{-1}}}{\|f(t)\|_{L^2}}.
\end{align}
In many situations it is expected (see \cite{MD18} and references therein) that $\lambda(t)\to c_{\nu}>0$, as $t\to \infty$ (possibly in
a time-averaged sense) and whenever $\nu>0$. Our result is not in contrast with this prediction, as it is conceivable to think that 
the constant $\eps_0$ appearing in \eqref{eq:engen} is in fact bigger than that appearing in \eqref{eq:enMIXmono}. Moreover, we
only provide upper bounds, which do not say much about the behavior of $\lambda(t)$. However, it is worth keeping in mind that
the exponent $1/3$ is optimal, and that when $\nu=0$ the result implies that  $\lambda(t)\to0$ as $t\to\infty$ whenever $f^{in}_{\neq}\neq0$.
\end{remark}

\subsection{Mixing and enhanced diffusion}
One of the main effects of transport is the generation of an averaging effect that creates smaller and smaller spatial scales as time increases.
Thinking on the Fourier side, this is equivalent to the accumulation of energy in higher and higher Fourier modes, as precisely described by
the term $\widehat{f^{in}}(\ell,\eta+\ell t)$ in \eqref{eq:couettesol} in the specific example of the Couette flow. It is therefore not surprising that
a way to quantify this is through the decay of a negative Sobolev norm, as in \eqref{eq:mixgen}. In the context of passive scalars, this
point of view was introduced in \cite{LTD11}, and it is deeply connected with the regularity of transport equations  
\cites{Jabin16,YZ17,Bressan03, CDL08}, the quantification and lower bounds on mixing rates  \cites{IKX14,Seis13,LLNMD12,ACM16,CLS17,Zillinger18, BCZ15}, and the inviscid damping in the 
two-dimensional Euler equations linearized around shear flows 
\cites{Zillinger14,Zillinger15,WZZ15,WZZkolmo17,WZZ17,GNRS18, CZZ18,Zcirc17,BCZV17}.

On the other hand, once a passive scalar is well 
concentrated at high frequencies, diffusion becomes the dominant effect. Understanding at which $\nu$-dependent time-scale 
this happens is fundamental in order  to capture correctly the dynamics. An enhanced diffusion estimate of the type \eqref{eq:engen}
precisely tells us that on a short time-scale $O(\nu^{-1/3})$, the solution becomes $x$-independent, while its $x$-average remains order 1
until the diffusive time-scale $O(\nu^{-1})$. This separation of time-scales has been the object of study in a vast part of the fluid mechanics
literature  \cites{BajerEtAl01,DubrulleNazarenko94,RhinesYoung83,LatiniBernoff01}. The rigorous mathematical study has been initiated
in \cite{CKRZ08}, and has become of great interest in the last years \cites{Deng2013,BGM15III,CKRZ08,BCZGH15,BCZ15,BMV14,BVW16,BGM15I,LiWeiZhang2017,BW13,Gallay2017,IMM17, CZDE18,
WZZkolmo17,BGM15II,WZ18,LWZ18}.

\subsection{Main ideas}
The proof of the enhanced diffusion estimate \eqref{eq:engen} is based on a modification of the so-called hypocoercivity method \cite{Villani09}.
In the context of fluid mechanics, it was successfully used for the first time in \cite{BW13} to study the Kolmogorov flow $u(y)=\sin y$, and then for general shear flows in \cite{BCZ15}
on a periodic domain or a bounded periodic channel. Other examples include circular flows \cite{CZD19}, the 2D Navier-Stokes 
equations \cites{CZEW19,WZ19} and the Boussinesq equations \cite{TW19}. Although the proof is very simple due to strict monotonicity,
a proof in the spatial domain $\T\times\R$ is not present in the literature, and we decided to include it here to make the paper self-contained.

Estimate \eqref{eq:enMIXmono} is the real novelty of this article. It relies on the observation that the vector field
\begin{align}\label{eq:vectorfie}
J=\de_y+tu'\de_x
\end{align}
commutes with the transport part of \eqref{eq:ADE}. Hence, the idea is to apply the hypocoercivity method not only to $f$, but also to $Jf$.
In view of the fact that the commutator $[J,\de_{yy}]\neq 0$, the scheme is much more involved and has to be carried out simultaneously
for $f$ and $Jf$. Once a proper enhanced diffusion estimate of the type \eqref{eq:engen} is proven for both $f$ and $Jf$, one simply
uses the fact that (see Lemma \ref{lem:boundbelow})
\begin{align}
t\| f \|_{\dot{H}^{-1}}\lesssim \|f\|_{L^2}+\|Jf\|_{L^2}.
\end{align}
The above inequality holds thanks to the monotonicity assumption in \eqref{eq:shear}. The vector field \eqref{eq:vectorfie} was used in 
the context of inviscid damping for the for the 2D $\beta$-plane equation \cite{WZZ18}.

\subsection{Notation and conventions}
All the proofs are carried out for the harder case of the hypoelliptic equation \eqref{eq:ADE}.
Throughout the paper, $\|\cdot\|$ and $\l\cdot,\cdot\r$ denote the standard $L^2$ norm and scalar product respectively. 
Via an expansion of $f$  as a Fourier series in the $x$ variable, namely
\begin{align}
f(t,x,y)=\sum_{\ell\in \Z} \ff(t,\ell,y)\e^{i\ell x}, \qquad \ff(t,\ell,y)=\frac{1}{2\pi}\int_\T f(t,x,y)\e^{-ikx}\dd x,
\end{align} 
for $k\in\N_0$ we set
\begin{equation}\label{eq:band}
 f_k(t,x,y):=\sum_{|\ell|=k} \ff(t,\ell,y)\e^{i \ell x}.
\end{equation}
This way we may express 
\begin{align}
f(t,x,y)=\sum_{k\in\N_0}f_k(t,x,y)
\end{align} 
as a sum of \emph{real-valued} functions $f_k$ that are localized in $x$-frequency on a single band $\pm k$, $k\in\N_0$. We thus see that \eqref{eq:ADE} decouples in $k$ and becomes an infinite system of one-dimensional equations. 

Given two operators $A,B$, the symbol $[A,B]=AB-BA$ denotes the commutator between the two. Unless otherwise stated, all the constants
will be \emph{independent} of $\nu$ and $k$, but may depend on $\UU$.

\section{Hypocoercivity for the passive scalar}
In this section, we tackle the question of enhanced diffusion for $f$. This is carried out by the so-called hypocoercivity method 
and the construction of a suitable energy functional. Let
\begin{align}\label{eq:alphabeta}
\alpha= \alpha_0 \frac{\nu^{2/3}}{k^{2/3}}, \qquad \beta= \beta_0 \frac{\nu^{1/3}}{k^{4/3}}, \qquad\gamma= \gamma_0 \frac{1}{k^2}.
\end{align}
where
\begin{align}\label{eq:alphabeta0}
\aJ_0= \frac{1}{4\times 3504 \UU^6},,\qquad \bJ_0=4\aJ_0^2,\qquad \cJ_0=128\aJ_0^3.
\end{align}
For each $k\in \N$, define the functional
\begin{align}\label{eq:Phifunc}
\Phi_k=\frac12\left[\| f_k\|^2+\alpha\|\de_{y} f_k\|^2+2\beta\l u'\de_x  f_k, \de_y f_k\r+\gamma\| u'\de_x  f_k\|^2\right],
\end{align}
where $f_k$ is given by \eqref{eq:band}. It is important to notice that the choice \eqref{eq:alphabeta} and \eqref{eq:alphabeta0}
implies that $\Phi_k$ is coercive. Indeed, since
\begin{align}
2\beta|\l u'\de_x  f_k, \de_y f_k\r|\leq 2\beta \|u'\de_x f_k\|\|\de_yf_k\|\leq \frac{\alpha}{2} \|\de_yf_k\|^2 +\frac{2\beta^2}{\alpha}\| u'\de_xf_k\|^2,
\end{align}
and, in particular,
\begin{align}\label{eq:constraint1}
\frac{\beta^2}{\alpha \gamma}\leq \frac14,
\end{align}
we find that
\begin{align}\label{eq:equiv}
\frac14\left[2\| f\|^2+\alpha\|\de_{y} f\|^2+\gamma \| u'\de_x  f\|^2\right] \leq\Phi_k\leq \frac14\left[2\| f\|^2+3\alpha\|\de_{y} f\|^2+3\gamma \| u'\de_x  f\|^2\right].
\end{align}
Then main result of this section is the derivation of a suitable differential inequality 
for $\Phi_k$, contained in the proposition below.

\begin{proposition}\label{prop:Phi-ODE}
Assume that $\nu k^{-1}\leq1$. For every $t\geq 0$, there holds
\begin{align}\label{eq:ODE5}
\ddt \Phi_k +2\eps_0\nu^{1/3}k^{2/3}\Phi_k
+\frac{\nu}{4}\|\de_{y}  f_k\|^2+\frac{\nu\alpha}{2}\|\de_{yy}  f_k\|^2+\frac{\nu\gamma}{2}\| u' \de_{xy} f_k\|^2\leq 0,
\end{align}
where
\begin{align}\label{eq:eps0}
\eps_0:=\frac{\beta_0 }{32\UU^2}
\end{align}
and $\beta_0$ is given in \eqref{eq:alphabeta0}.
\end{proposition}
As we shall see in a moment, the choice of the coefficients \eqref{eq:alphabeta0} and \eqref{eq:eps0} is an overkill in the
proof of Proposition \ref{prop:Phi-ODE}, but will be necessary for Proposition \ref{prop:J-ODE}. Since we do not want to
make two different choices in the parameters, we opted for a unified selection.

\subsection{Energy balances}\label{sub:balance}
We begin by computing the time derivative of $\Phi_k$. At this stage, we omit the subscript in all the functions, keeping in mind
the useful relation $\|\de_x f\|=k\|f\|$ in the estimates. Relying on the antisymmetry of the transport term, we test \eqref{eq:ADE}
by $f$ and derive the energy balance
\begin{align}\label{eq:energy}
\frac12\ddt\| f\|^2+\nu\|\de_{y}  f\|^2=0.
\end{align}
Regarding the $L^2$ balance for $\de_yf$, we obtain
\begin{align}\label{eq:alpha}
\frac12\ddt\|\de_{y} f\|^2+\nu\|\de_{yy}  f\|^2=-\l u'\de_x  f, \de_y f\r,
\end{align}
where the right-hand side above appears because of the commutator relation $[\de_y,u\de_x]=u'\de_x$. Turning
to the $\beta$-term in \eqref{eq:Phifunc}, we have
\begin{align}\label{eq:beta}
\ddt\l u'\de_x  f, \de_y f\r
&=\nu\left[ \l u'\de_{xyy} f , \de_y f\r+ \l u'\de_{x} f , \de_{yyy} f\r\right]-\l u'\de_x  (u\de_x f), \de_y f\r-\l u'\de_x  f, \de_y (u\de_xf)\r\notag\\
&=-2\nu\l u'\de_{xy}  f, \de_{yy} f\r-\nu\l u''\de_x  f, \de_{yy} f\r        - \| u'\de_x  f\|^2,
\end{align}
where we integrated by parts, used the antisymmetry of $u''\de_x$ and the crucial relation $[u\de_x,u'\de_x]=0$.
Finally, the time derivative of the $\gamma$-term in  \eqref{eq:Phifunc} can be computed as
\begin{align}
\frac12\ddt\| u'\de_x  f\|^2
&=\nu \l u'\de_{x},u'\de_{xyy}  f\r   -  \l u'\de_{x},u'\de_{x} (u\de_x  f)\r\notag\\
&=- \nu\| u' \de_{xy} f\|^2-2\nu \l u'u''\de_{x}f,\de_{xy}  f\r\notag \\
&=- \nu\| u' \de_{xy} f\|^2+\nu \l (u'u'')'\de_{x}f,\de_{x}  f\r.
\end{align}
Therefore,
\begin{align}\label{eq:gamma}
\frac12\ddt\| u'\de_x  f\|^2+\nu\| u' \de_{xy} f\|^2=\nu\l (u'u'')'\de_x f,\de_xf\r.
\end{align}
Collecting \eqref{eq:energy}, \eqref{eq:alpha}, \eqref{eq:beta} and \eqref{eq:gamma} we obtain
\begin{align}\label{eq:Phidiffeq}
&\ddt \Phi+\nu\|\de_{y}  f\|^2+\nu\alpha\|\de_{yy}  f\|^2
 +\beta \| u'\de_x  f\|^2+\nu\gamma\| u'\de_{xy} f\|^2=-\alpha\l u'\de_x  f, \de_y f\r\notag\\
 &\qquad\qquad-2\nu\beta\l u'\de_{xy}  f, \de_{yy} f\r-\nu\beta\l u''\de_x  f, \de_{yy} f\r+\nu\gamma\l (u'u'')'\de_x f,\de_xf\r.
\end{align}
We now proceed in estimating the error terms above, one by one.

\subsection{Error estimates and proof of Proposition \ref{prop:Phi-ODE}}
In what follows, we will repeatedly use the Cauchy-Schwarz and Young inequalities without mention.
We begin from the $\alpha$-term in the right-hand side of \eqref{eq:Phidiffeq}. We have
\begin{align}\label{eq:err1}
\alpha|\l u'\de_x  f, \de_y f\r|\leq \alpha \| u'\de_x  f \| \| \de_y f\| \leq  \frac{\nu}{2} \|\de_y f\|^2+\frac{\alpha^2}{2\nu} \| u'\de_x  f \|^2.
\end{align}
Since by \eqref{eq:alphabeta}-\eqref{eq:alphabeta0} we have
\begin{align}\label{eq:constraint2}
\frac{\alpha^2}{\beta\nu}\leq \frac{1}{2},
\end{align}
we find
\begin{align}\label{eq:err1sol}
\alpha|\l u'\de_x  f, \de_y f\r| \leq  \frac{\nu}{2} \|\de_y f\|^2+ \frac{\beta}{4} \| u'\de_x  f \|^2.
\end{align}
For the first of the $\beta$-terms in  \eqref{eq:Phidiffeq}, we have
\begin{align}\label{eq:err2}
2\nu\beta|\l u'\de_{xy}  f, \de_{yy} f\r|\leq 2\nu\beta\|u'\de_{xy}  f\|\| \de_{yy} f\|\leq \frac{\nu \alpha}{4}\|\de_{yy} f\|^2+ \frac{4\nu\beta^2}{\alpha}\|u'\de_{xy}  f\|^2,
\end{align}
so that since \eqref{eq:alphabeta0} implies that
\begin{align}\label{eq:constraint1bis}
\frac{\beta^2}{\alpha\gamma}= \frac{1}{8},
\end{align}
we obtain
\begin{align}\label{eq:err2sol}
2\nu\beta|\l u'\de_{xy}  f, \de_{yy} f\r|\leq \frac{\nu\alpha}{4}\|\de_{yy} f\|^2+ \frac{\gamma\nu}{2}\|u'\de_{xy}  f\|^2.
\end{align}
The second $\beta$-term can be instead estimated as
\begin{align}\label{eq:err3}
\nu\beta|\l u''\de_x  f, \de_{yy} f\r|\leq  \nu\beta\| u''\de_x  f\| \|\de_{yy} f\|\leq  \frac{\nu\alpha}{4}\|\de_{yy} f\|^2+\frac{\nu\beta^2}{\alpha}\| u''\de_x  f\|^2,
\end{align}
so that using \eqref{eq:constraint1bis} and \eqref{eq:shear} we have
\begin{align}\label{eq:err3sol}
\nu\beta|\l u''\de_x  f, \de_{yy} f\r|\leq \frac{\nu\alpha}{4}\|\de_{yy} f\|^2+\UU^2\frac{\nu\gamma}{8}\| u'\de_x  f\|^2.
\end{align}
Lastly, we have
\begin{align}\label{eq:err4sol}
\nu\gamma|\l (u'u'')'\de_x f,\de_xf\r|\leq 2\nu\UU^2\gamma \| u'\de_x  f \|^2.
\end{align}
Collecting \eqref{eq:err1sol}, \eqref{eq:err2sol}, \eqref{eq:err3sol}, \eqref{eq:err4sol} and using \eqref{eq:Phidiffeq}
we arrive at
\begin{align}\label{eq:ODE1}
&\ddt \Phi+\frac{\nu}{2}\|\de_{y}  f\|^2+\frac{\nu\alpha}{2}\|\de_{yy}  f\|^2
 +\frac{3\beta}{4} \| u'\de_x  f\|^2+\frac{\nu\gamma}{2}\| u'\de_{xy} f\|^2\leq 3\UU^2\nu\gamma \| u'\de_x  f \|^2.
\end{align}
Looking at the right-hand side of \eqref{eq:ODE1}, recalling that $\nu k^{-1}\leq1$ and using  \eqref{eq:alphabeta}, we have
\begin{align}
 3\UU^2\nu\gamma \| u'\de_x  f \|^2= 3\UU^2\frac{\gamma_0}{\beta_0}\left(\frac{\nu}{k}\right)^{2/3}\beta \| u'\de_x  f \|^2\leq \frac{\beta}{4} \| u'\de_x  f \|^2,
\end{align}
where we used that  from \eqref{eq:alphabeta0}  we have
\begin{align}\label{eq:constraint3}
\frac{\gamma_0}{\beta_0}=32\alpha_0\leq \frac{1}{12\UU^2}.
\end{align}
Thus, \eqref{eq:ODE1}  becomes
\begin{align}\label{eq:ODE2}
&\ddt \Phi+\frac{\nu}{2}\|\de_{y}  f\|^2+\frac{\nu\alpha}{2}\|\de_{yy}  f\|^2
 +\frac{\beta}{2} \| u'\de_x  f\|^2+\frac{\nu\gamma}{2}\| u' \de_{xy} f\|^2\leq 0.
\end{align}
In view of \eqref{eq:alphabeta} and \eqref{eq:shear}, this implies
\begin{align}
\ddt \Phi +\frac{\beta k^2}{4\UU^2} \|   f\|^2+\frac{\nu}{2}\|\de_{y}  f\|^2+\frac{\beta}{4} \| u'\de_x  f\|^2+\frac{\nu\alpha}{2}\|\de_{yy}  f\|^2+\frac{\nu\gamma}{2}\| u' \de_{xy} f\|^2\leq 0,
\end{align}
or, equivalently,
\begin{align}
&\ddt \Phi +\frac{\beta_0 }{8\UU^2}\nu^{1/3}k^{2/3}\left[2 \|f\|^2+\frac{2\UU^2}{\alpha_0\beta_0}\alpha\|\de_y  f\|^2+\frac{2\UU^2}{\gamma_0}\gamma\| u'\de_x  f\|^2\right]\notag\\
&\qquad\qquad+\frac{\nu}{4}\|\de_{y}  f\|^2+\frac{\nu\alpha}{2}\|\de_{yy}  f\|^2+\frac{\nu\gamma}{2}\| u' \de_{xy} f\|^2\leq 0.
\end{align}
In light of \eqref{eq:alphabeta0}, we have
\begin{align}\label{eq:constraint4}
\frac{2\UU^2}{\alpha_0\beta_0}\geq 3,\qquad \frac{2\UU^2}{\gamma_0}\geq 3.
\end{align}
Using \eqref{eq:equiv} we have
\begin{align}\label{eq:ODE4}
\ddt \Phi +\frac{\beta_0 }{4\UU^2}\nu^{1/3}k^{2/3}\Phi
+\frac{\nu}{4}\|\de_{y}  f\|^2+\frac{\nu\alpha}{2}\|\de_{yy}  f\|^2+\frac{\nu\gamma}{2}\| u' \de_{xy} f\|^2\leq 0.
\end{align}
In particular, this implies \eqref{eq:ODE5}, and the proof of Proposition \ref{prop:Phi-ODE} is concluded.

\section{Hypocoercivity and the vector field method}
Let us define the vector field
\begin{align}\label{eq:vectorfield}
J=\de_y+tu'\de_x,
\end{align}
and, for each $k\in\N$, the corresponding energy functional
\begin{align}\label{eq:Jfunc}
\J_k=\frac12\left[\| Jf_k\|^2+\aJ\|\de_y Jf_k\|^2+2\bJ\l u'\de_x  Jf_k, \de_y Jf_k\r+\cJ\| u'\de_x  Jf_k\|^2\right],
\end{align}
where the various coefficients are exactly the same as in \eqref{eq:alphabeta}-\eqref{eq:alphabeta0}. Thus, in particular, there holds
\begin{align}\label{eq:equivJ}
\frac14\left[2\| Jf_k\|^2+\alpha\|\de_{y}J f_k\|^2+\gamma \| u'\de_x  Jf_k\|^2\right] \leq\J_k\leq \frac14\left[2\|J f_k\|^2+3\alpha\|\de_{y} Jf_k\|^2+3\gamma \| u'\de_x J f_k\|^2\right].
\end{align}
The goal here is to prove a result analogous to Proposition \ref{prop:Phi-ODE}, but with a slightly more complicated form.

\begin{proposition}\label{prop:J-ODE}
Assume that $\nu k^{-1}\leq\nu_0$, where
\begin{align}\label{eq:nukrestr}
\nu_0:=  \left(\frac{\bJ_0}{4\times 7008 \UU^{8}}\right)^{3/2}.
\end{align}
For every $t\geq 0$, there holds
\begin{align}\label{eq:ODE11J}
\ddt \J_k +2\eps_0\nu^{1/3}k^{2/3} \J_k
\leq3504\nu \UU^{6}\left[ \|\de_yf_k\|^2 + \aJ \|\de_{yy}f_k\|^2+\cJ \|u'\de_{xy}f_k\|^2\right],
\end{align}
where $\eps_0$ is the same as in \eqref{eq:eps0} and $\beta_0$ is given in \eqref{eq:alphabeta0}.
\end{proposition}
The proof of the above result is vaguely reminiscent of  that of Proposition  \ref{prop:Phi-ODE}, and it will be carried out in the next sections.

\subsection{The vector field and inviscid mixing estimates}
Before going to the proof, let us clarify the reason of the choice of the vector field in \eqref{eq:vectorfield}. We apply $J$ to
the advection diffusion equation \eqref{eq:ADE}.
Since
\begin{align}
[J,\de_t+u\de_{x}]=0
\end{align}
and
\begin{align}
[J,\de_{yy}]=- tu'''\de_x -2 tu'' \de_{xy}
&=-\frac{u'''}{u'} (J-\de_y)-2 u'' \de_{y}\left(\frac{1}{u'}(J-\de_y)\right)\notag\\
&=\frac{u'''}{u'} (J-\de_y)-2 \de_{y}\left(  \frac{u''}{u'}(J-\de_y)\right),
\end{align}
we have that
\begin{align}\label{eq:vectoreq}
\de_t Jf+u\de_x J f=\nu \de_{yy} Jf +\nu\frac{u'''}{u'} (Jf-\de_yf)-2 \nu\de_{y}\left(  \frac{u''}{u'}(Jf-\de_yf)\right).
\end{align}
It is clear from the above computation why the strict monotonicity requirement on $u$ is important. The crucial property of $J$
is that, combined with the $L^2$ norm of $f_k$, it provides a bound on the $\dot{H}^{-1}$ norm of $f_k$ itself, with an extra factor of $t$.
This is essentially contained in \cite{WZZ18}.

\begin{lemma}\label{lem:boundbelow}
Let $k\in \N$. There there holds
\begin{align}\label{eq:boundbelow}
kt\| f_k(t) \|_{\dot{H}^{-1}}\leq 2\UU^2\left[\|f_k(t)\|+\|Jf_k(t)\|\right],
\end{align}
for every $t>0$.
\end{lemma}

\begin{proof}[Proof of Lemma \ref{lem:boundbelow}]
For each $k\in \N$, let us define the ``stream-function'' $\psi_k$ as the solution to the elliptic problem
\begin{align}
\Delta\psi_k=f_k,
\end{align}
with appropriate periodic boundary conditions in $x$, decay in $y$, and mean zero condition. Then, a series of 
integration by parts entails
\begin{align}
t\| \nabla \de_x\psi_k \|^2&= -t\l \de_x\psi_k, \de_xf_k\r=-\l \de_x\psi_k, \frac{1}{u'} t u'\de_xf_k\r=
-\l \de_x\psi_k, \frac{1}{u'} (Jf_k-\de_yf_k)\r\notag\\
&=-\l \de_x\psi_k, \frac{1}{u'} Jf_k\r+\l \frac{1}{u'}\de_x\psi_k, \de_yf_k\r
=-\l \de_x\psi_k, \frac{1}{u'} Jf_k\r-\l \de_y(\frac{1}{u'}\de_x\psi_k), f_k)\r\notag\\
&\leq \UU \|\de_x\psi_k\|\|Jf_k\| +\UU\left[\|\de_{xy}\psi_k\| +\UU\|\de_x\psi_k\| \right] \|f_k\|.
\end{align}
Thus
\begin{align}
t\| \nabla \de_x\psi_k \|\leq 2\UU^2\left[\|f_k\|+\|Jf_k\|\right],
\end{align}
which readily implies \eqref{eq:boundbelow} and concludes the proof.
\end{proof}

The power of this technique can be understood by looking at \eqref{eq:vectoreq} for $\nu=0$. In that case, we find
\begin{align}
\ddt \left[\|f_k\|^2+\|J f_k\|^2\right]=0,
\end{align}
and since $Jf_k(0)=\de_yf_k^{in}$, we find from \eqref{eq:boundbelow} that
\begin{align}
(kt)^2\| f_k(t) \|^2_{\dot{H}^{-1}}\leq 4\UU^4\left[\|f^{in}_k\|^2+\|\de_yf_k^{in}\|^2\right].
\end{align}
This implies \eqref{eq:mixgen}, and in fact it does not even require the weighted $L^2$ norm. 
However, the proof of \eqref{eq:enMIXmono} is much more involved.

\subsection{Energy balances}
As in Section \ref{sub:balance}, we omit the dependence on $k$ of the various functions.
We begin by testing \eqref{eq:vectoreq}  with $Jf$. We have
\begin{align}
\frac12\ddt \| Jf\|^2 +\nu\|\de_y Jf\|^2
&=\nu \left[ \l \frac{u'''}{u'} (Jf-\de_yf),Jf\r -2 \l\de_{y}\left(  \frac{u''}{u'}(Jf-\de_yf) \right),Jf\r\right]
\end{align}
Thus, integrating by parts the last term entails
\begin{align}\label{eq:energyJ}
\frac12\ddt \| Jf\|^2 +\nu\|\de_y Jf\|^2=\nu E_0
\end{align}
where
\begin{align}\label{eq:E0}
E_0= \l \frac{u'''}{u'} (Jf-\de_yf),Jf\r +2 \l  \frac{u''}{u'}(Jf-\de_yf) ,\de_yJf\r.
\end{align}
Taking $\de_y$ of  \eqref{eq:vectoreq}, we have
\begin{align}\label{eq:vectoreqDY}
\de_t \de_yJf+u\de_{xy} J f=\nu \de_{yyy}Jf +\nu\de_y\left(\frac{u'''}{u'} (Jf-\de_yf)\right)-2 \nu\de_{yy}\left(  \frac{u''}{u'}(Jf-\de_yf)\right)- u' \de_x  J f.
\end{align}
Thus,  multiplying by $\de_yJf$ in $L^2$, we have
\begin{align}
\frac12\ddt \| \de_y Jf\|^2 +\nu\|\de_{yy} Jf\|^2=&\nu \left[ \l\de_y\left(\frac{u'''}{u'} (Jf-\de_yf)\right),\de_y Jf\r-2 \l\de_{yy}\left(  \frac{u''}{u'}(Jf-\de_yf)\right),\de_y Jf\r  \right]\notag\\
&-\l u' \de_x  J f, \de_yJf\r\notag\\
=&-\nu \left[ \l\frac{u'''}{u'} (Jf-\de_yf),\de_{yy} Jf\r-2 \l\de_{y}\left(  \frac{u''}{u'}(Jf-\de_yf)\right), \de_{yy} Jf\r  \right]\notag\\
&-\l u' \de_x  J f, \de_yJf\r
\end{align}
As a consequence,
\begin{align}\label{eq:alphaJ}
\frac12\ddt \| \de_y Jf\|^2 +\nu\|\de_{yy} Jf\|^2=E_{1},
\end{align}
having defined
\begin{align}\label{eq:E1}
E_{1}=&\nu \left[ \l\frac{u' u'''+2(u'')^2}{(u')^2} (Jf-\de_yf),\de_{yy} Jf\r-2 \l  \frac{u''}{u'}(\de_{y}Jf-\de_{yy}f), \de_{yy} Jf\r  \right]\notag\\
&-\l u' \de_x  J f, \de_yJf\r.
\end{align}
Turning to the $\beta$-term in \eqref{eq:Jfunc}, using \eqref{eq:vectoreq} and several integration by parts we have
\begin{align}
&\ddt\l u' \de_x  J f, \de_yJf\r+\|u'\de_x Jf\|^2\notag\\
&\quad=\nu \bigg[
\l u' \de_{xyy} Jf  , \de_yJf\r +\l u' \de_x  J f, \de_{yyy} Jf\r\notag\\
&\quad\qquad+ \l u' \de_x \left(\frac{u'''}{u'} (Jf-\de_yf)\right), \de_yJf\r+\l u' \de_x  J f, \de_y \left(\frac{u'''}{u'} (Jf-\de_yf)\right) \r\notag\\
&\quad\qquad-2\l u'  \de_{xy}\left(  \frac{u''}{u'}(Jf-\de_yf)\right), \de_yJf\r-2\l u' \de_x  J f, \de_{yy}\left(  \frac{u''}{u'}(Jf-\de_yf)\right)\r \bigg]\notag\\
&\quad=\nu \bigg[
-2\l u' \de_{xy}Jf,   \de_{yy} Jf \r-\l u'' \de_x  J f, \de_{yy} Jf\r - 2\l   u''' \de_xJf, \de_yJf\r\notag\\
&\quad\qquad -4\l    u''\de_{yy}f, \de_{xy}Jf\r -4\l  \frac{u'u'''-(u'')^2}{u'}\de_x Jf, \de_{y}Jf\r-\l  \frac{6u'u'''-4(u'')^2}{u'}\de_yf, \de_{xy}Jf\r\notag\\
&\quad\qquad+2\l u'' \de_x  J f,   \frac{u''}{u'}(\de_yJf-\de_{yy}f)\r-\l u'' \de_x  J f,  \frac{u'u'''-2(u'')^2}{(u')^2}\de_yf\r \bigg],
%
%
%
\end{align}
implying
\begin{align}\label{eq:betaJ}
\ddt\l u' \de_x  J f, \de_yJf\r +\|u'\de_x Jf\|^2 &=\nu E_2,
\end{align}
for $E_2$ defined as
\begin{align}\label{eq:E2}
E_2&= \bigg[
-2\l u' \de_{xy}Jf,   \de_{yy} Jf \r-\l u'' \de_x  J f, \de_{yy} Jf\r - 2\l   u''' \de_xJf, \de_yJf\r\notag\\
&\quad\qquad -4\l    u''\de_{yy}f, \de_{xy}Jf\r -4\l  \frac{u'u'''-(u'')^2}{u'}\de_x Jf, \de_{y}Jf\r-\l  \frac{6u'u'''-4(u'')^2}{u'}\de_yf, \de_{xy}Jf\r\notag\\
&\quad\qquad+2\l u'' \de_x  J f,   \frac{u''}{u'}(\de_yJf-\de_{yy}f)\r-\l u'' \de_x  J f,  \frac{u'u'''-2(u'')^2}{(u')^2}\de_yf\r \bigg].
\end{align}
Lastly, for the $\cJ$-term we argue as in \eqref{eq:gamma} and obtain
\begin{align}
&\frac12\ddt\| u'\de_x J f\|^2+\nu\| u' \de_{xy} Jf\|^2\notag\\
&\qquad\qquad=\nu\bigg[\l (u'u'')'\de_x Jf,\de_xJf\r+ \l u''' \de_x(Jf-\de_yf),u'\de_xJf\r \notag\\
&\qquad\qquad\qquad\quad-2\l u'\de_{xy}\left(  \frac{u''}{u'}(Jf-\de_yf)\right),u'\de_xJf\r\bigg]\notag\\
&\qquad\qquad=\nu\bigg[\l (4 (u'')^2+u'u''')\de_{x}Jf,\de_xJf\r -\l (4 (u'')^2+u'u''')\de_{xy}f,\de_xJf\r \notag\\
&\qquad\qquad\qquad\quad-2\l u''\de_{xy}f,u'\de_{xy}Jf\r\bigg].
\end{align}
Therefore,
\begin{align}\label{eq:gammaJ}
\frac12\ddt\| u'\de_x J f\|^2+\nu\| u' \de_{xy} Jf\|^2=\nu E_3,
\end{align}
where
\begin{equation}\label{eq:E3}
E_3=\l (4 (u'')^2+u'u''')\de_{x}Jf,\de_xJf\r -\l (4 (u'')^2+u'u''')\de_{xy}f,\de_xJf\r -2\l u''\de_{xy}f,u'\de_{xy}Jf\r.
\end{equation}
Collecting \eqref{eq:energyJ}, \eqref{eq:alphaJ}, \eqref{eq:betaJ} and \eqref{eq:gammaJ} we end up with
\begin{align}\label{eq:ODE1J}
&\ddt \J+\nu\|\de_y  Jf\|^2+\nu\aJ\|\de_{yy}  Jf\|^2
 +\bJ \| u'\de_x  Jf\|^2+\cJ\nu\| u' \de_{xy} Jf\|^2\notag\\
 &\qquad=\nu E_0+\aJ E_1+\nu \bJ E_2+\nu \cJ E_3.
  \end{align}
To reach \eqref{eq:ODE11J}, we now have to estimate all the error terms in a suitable way.

\subsection{Error estimates}
In what follows, we will make heavy use of the assumption \eqref{eq:shear}, keeping track of the dependence on $\UU$ of all the constants.
The error term $E_0$ in \eqref{eq:E0} is easily estimated as
\begin{align}\label{eq:ERRL2J}
|E_0|\leq \frac{1}{4}\|\de_yJf\|^2+10\UU^2\|Jf\|^2+9\UU^2\|\de_yf\|^2.
\end{align}
Turning to $E_1$ in \eqref{eq:E1}, we notice that the last term is essentially estimated as in \eqref{eq:err1sol}, exploiting the fact
that
\begin{align}\label{eq:constraint2J}
\frac{\aJ^2}{\bJ\nu}= \frac{1}{4},
\end{align}
and obtaining
\begin{align}
\aJ \left|\l u' \de_x  J f, \de_yJf\r\right|\leq  \frac{\nu}{4} \|\de_y Jf\|^2+ \frac{\bJ}{4} \| u'\de_x J f \|^2.
\end{align}
As for the other terms,
\begin{align}
& \left| \l\frac{u' u'''+2(u'')^2}{(u')^2} (Jf-\de_yf),\de_{yy} Jf\r-2 \l  \frac{u''}{u'}(\de_{y}Jf-\de_{yy}f), \de_{yy} Jf\r  \right|\notag\\
&\qquad\leq \UU\left[3\UU\|Jf\|\|\de_{yy} Jf\|+3\UU\|\de_yf\|\|\de_{yy} Jf\| +2 \|\de_yJf\|\|\de_{yy} Jf\|+2\|\de_{yy}f\| \|\de_{yy} Jf\| \right]\notag\\
&\qquad\leq\frac{1}{4}  \|\de_{yy} Jf\|^2 + 36  \UU^4\|Jf\|^2+ 36  \UU^4\|\de_yf\|^2+16 \UU^2 \|\de_yJf\|^2+16\UU^2\|\de_{yy}f\|^2.
\end{align}
Thus,
\begin{align}\label{eq:ERRH1J}
\alpha |E_{1}|
&\leq  \frac{\nu}{4} \|\de_y Jf\|^2+ \frac{\bJ}{4} \| u'\de_x J f \|^2
+\frac{\aJ\nu}{4}  \|\de_{yy} Jf\|^2\notag \\ 
&\quad + 36 \aJ\nu \UU^4\|Jf\|^2+ 36 \aJ\nu \UU^4\|\de_yf\|^2+16 \aJ \nu\UU^2 \|\de_yJf\|^2+16 \aJ\nu\UU^2\|\de_{yy}f\|^2.
\end{align}
We now deal with $E_2$ in \eqref{eq:E2}. Like in \eqref{eq:err2sol} and \eqref{eq:err3sol}, recalling that
\begin{align}\label{eq:constraint1bisJ}
\frac{\bJ^2 }{\aJ\cJ}= \frac{1}{8},
\end{align}
we have
\begin{align}\label{eq:err2solJ}
2\nu\bJ|\l u'\de_{xy}  Jf, \de_{yy} Jf\r|\leq \frac{\nu\aJ}{4}\|\de_{yy} Jf\|^2+ \frac{\nu \cJ}{2}\|u'\de_{xy}  Jf\|^2,
\end{align}
and
\begin{align}\label{eq:err3solJ}
\nu\bJ|\l u''\de_x  Jf, \de_{yy} Jf\r|\leq \frac{\nu\aJ}{4}\|\de_{yy} Jf\|^2+\frac{\nu\cJ}{8}\| u''\de_x  Jf\|^2.
\end{align}
The other terms are estimated as
\begin{align}
&\nu \beta  \bigg|
 - 2\l   u''' \de_xJf, \de_yJf\r-4\l    u''\de_{yy}f, \de_{xy}Jf\r -4\l  \frac{u'u'''-(u'')^2}{u'}\de_x Jf, \de_{y}Jf\r\notag\\
&\quad -\l  \frac{6u'u'''-4(u'')^2}{u'}\de_yf, \de_{xy}Jf\r+2\l u'' \de_x  J f,   \frac{u''}{u'}(\de_yJf-\de_{yy}f)\r\notag\\
&\quad-\l u'' \de_x  J f,  \frac{u'u'''-2(u'')^2}{(u')^2}\de_yf\r   \bigg|\notag\\
&\qquad\leq \beta\nu\Big[ 12\UU^2 \|u'\de_xJf\|\|\de_yJf\| +4\UU\| u'\de_{xy}Jf\|\|\de_{yy}f\|
+10 \UU^2\|u'\de_{xy}Jf\|\|\de_yf\|\notag\\
&\qquad\qquad\quad+2\UU^2\|u'\de_xJf\|\|\de_{yy}f\| +3\UU^3 \|u'\de_xJf\|\|\de_yf\|\Big]\notag\\
&\qquad\leq \frac{\bJ}{4}\| u'\de_xJf\|^2+\frac{\cJ\nu}{8} \|u'\de_{xy}Jf\|^2+1152\bJ\nu^2 \UU^4\|\de_yJf\|^2+\frac{400\bJ^2 \nu \UU^4}{\cJ}\|\de_yf\|^2\notag\\
 &\qquad\quad+\frac{64 \bJ^2 \nu \UU^2}{\cJ} \|\de_{yy}f\|^2 + 27 \bJ\nu^2\UU^{6}\|\de_yf\|^2
+12\bJ\nu^2\UU^4\|\de_{yy}f\|^2.
\end{align}
Taking into account \eqref{eq:constraint1bisJ} and collecting \eqref{eq:err2solJ}, \eqref{eq:err3solJ} and the estimates above, we have
\begin{align}\label{eq:ERRcrossJ}
\nu \beta|E_{2}|
&\leq \frac{\nu\aJ}{2}\|\de_{yy} Jf\|^2+ \frac{5\nu \cJ}{8}\|u'\de_{xy}  Jf\|^2 +\frac{\bJ}{4}\| u'\de_xJf\|^2\notag\\
&\quad +\UU^2\frac{\nu\cJ}{8}\| u'\de_x  Jf\|^2 +1152\bJ\nu^2 \UU^4\|\de_yJf\|^2 +50\aJ \nu \UU^4\|\de_yf\|^2+8 \aJ \nu \UU^2 \|\de_{yy}f\|^2  \notag\\
&\quad+ 27 \bJ\nu^2\UU^{6}\|\de_yf\|^2
+12\bJ\nu^2\UU^4\|\de_{yy}f\|^2.
\end{align}
Lastly, looking at \eqref{eq:E3}, we have
\begin{align}\label{eq:ERRgammaJ}
|E_3|&=\l (4 (u'')^2+u'u''')\de_{x}Jf,\de_xJf\r -\l (4 (u'')^2+u'u''')\de_{xy}f,\de_xJf\r -2\l u''\de_{xy}f,u'\de_{xy}Jf\r\notag\\
&\leq 5\UU^2\|u'\de_{x}Jf\|^2+5\UU^2\|u'\de_{x}Jf\|\| u'\de_{xy}f\|+ \UU\|u'\de_{xy}Jf\|\| u'\de_{xy}f\|\notag\\
&\leq \frac18\|u'\de_{xy}Jf\|^2+7\UU^2\|u'\de_{x}Jf\|^2+8\UU^2\| u'\de_{xy}f\|^2.
\end{align}
Now that all the errors have been properly controlled, we are ready to prove Proposition \ref{prop:J-ODE}.

\subsection{Proof of Proposition \ref{prop:J-ODE}}
Collecting \eqref{eq:ERRL2J}, \eqref{eq:ERRH1J}, \eqref{eq:ERRcrossJ} and  \eqref{eq:ERRgammaJ} into \eqref{eq:ODE1J} we have
\begin{align}\label{eq:ODE2J}
&\ddt \J+\frac{\nu}{2}\|\de_y  Jf\|^2+\frac{\nu\aJ}{4}\|\de_{yy}  Jf\|^2+\frac{\bJ}{2} \| u'\de_x  Jf\|^2+\frac{\cJ\nu}{4}\| u' \de_{xy} Jf\|^2\notag\\
 &\qquad\leq 1168\UU^6\nu\big[(1+\aJ) \| Jf\|^2+ (\aJ+\bJ\nu )\| \de_yJf\|^2 + \cJ\|u' \de_{x}Jf\|^2  \notag\\
 &\qquad\quad+(1+\aJ+\bJ\nu) \|\de_yf\|^2 + (\aJ+\bJ\nu) \|\de_{yy}f\|^2+\cJ \|u'\de_{xy}f\|^2\big].
  \end{align}
Recall that $\nu k^{-1}\leq 1$ and that \eqref{eq:alphabeta}-\eqref{eq:alphabeta0} in particular implies that
\begin{align}\label{eq:constraint4J}
\bJ_0\leq\aJ_0\leq 1.
\end{align}
Computing all the coefficients of the above errors, we have
\begin{align}
1+\aJ\leq 2, \qquad \aJ+\bJ\nu\leq2\aJ\leq 2\aJ_0,\qquad 1+\aJ+\bJ\nu\leq3,
\end{align}
so that
\begin{align}\label{eq:ODE3J}
&\ddt \J+\frac{\nu}{2}\|\de_y  Jf\|^2+\frac{\nu\aJ}{4}\|\de_{yy}  Jf\|^2+\frac{\bJ}{2} \| u'\de_x  Jf\|^2+\frac{\cJ\nu}{4}\| u' \de_{xy} Jf\|^2\notag\\
 &\qquad\leq 3504\UU^6\nu\big[\| Jf\|^2+ \aJ_0\| \de_yJf\|^2 + \cJ\|u' \de_{x}Jf\|^2 +\|\de_yf\|^2 + \aJ \|\de_{yy}f\|^2+\cJ \|u'\de_{xy}f\|^2\big].
  \end{align}
In light of the restriction $\nu k^{-1}\leq \nu_0$ given by \eqref{eq:nukrestr}, we have
\begin{align}
3504\nu  \UU^{6}\left[\| Jf\|^2+ \cJ\|u' \de_{x}Jf\|^2\right]\leq 
\frac{7008 \UU^{8}}{\bJ_0}\frac{\nu^{2/3}}{k^{2/3}} \bJ\|  u'\de_x Jf\|^2\leq \frac{\bJ}{4}\|  u'\de_x Jf\|^2,
\end{align}
Using once more the expression of $\aJ_0$ in \eqref{eq:alphabeta0}, \eqref{eq:ODE2J} becomes
\begin{align}\label{eq:ODE5J}
&\ddt \J+\frac{\nu}{4}\|\de_y  Jf\|^2+\frac{\nu\aJ}{4}\|\de_{yy}  Jf\|^2+\frac{\bJ}{4} \| u'\de_x  Jf\|^2+\frac{\cJ\nu}{4}\| u' \de_{xy} Jf\|^2\notag\\
 &\qquad\leq 3504\UU^6\nu\big[\|\de_yf\|^2 + \aJ \|\de_{yy}f\|^2+\cJ \|u'\de_{xy}f\|^2\big].
  \end{align}
We now proceed as we did after \eqref{eq:ODE2}, arriving after a few easy computations at
\begin{align}\label{eq:ODE6J}
&\ddt \J +\frac{\bJ_0 }{16\UU^2}\nu^{1/3}k^{2/3} \J+\frac{\nu\aJ}{4}\|\de_{yy}  Jf\|^2+\frac{\cJ\nu}{4}\| u' \de_{xy} Jf\|^2\notag\\
 &\qquad\leq 3504\UU^6\nu\big[\|\de_yf\|^2 + \aJ \|\de_{yy}f\|^2+\cJ \|u'\de_{xy}f\|^2\big].
\end{align}
This is equivalent to \eqref{eq:ODE11J}, and therefore the proof is over.

\section{Weighted estimates in the energy norm}
The Fourier-localized version of Theorem \ref{thm:main} is the following result.
\begin{theorem}\label{thm:mainFour}
There exist $C_0\geq 1$ and $\eps_0,\nu_0\in(0,1)$ only depending on $\UU$ such that if $\nu k^{-1}\in[0,\nu_0]$ we have the enhanced diffusion estimate
\begin{align}\label{eq:engenFour2}
\|f_{k}(t)\|_{u'}\leq C_0 \e^{-\eps_0 \nu^{1/3}k^{2/3} t}\|f^{in}_{k}\|_{u'},\qquad \forall t\geq 0,
\end{align}
and the stable mixing estimate
\begin{align}\label{eq:enMIXmonoFour}
\|f_{k}(t)\|_{\dot{H}^{-1}}\leq C_0\frac{ \e^{-\eps_0 \nu^{1/3}k^{2/3} t}}{\sqrt{1+(kt)^2}}\left[\|f^{in}_{k}\|_{u'}+\|\de_y f^{in}_{k}\|_{u'}\right],\qquad  \forall t\geq 0.
\end{align}
In particular, if $\nu=0$ we obtain the inviscid mixing estimate
\begin{align}\label{eq:mixgenFour}
\|f_{k}(t)\|_{\dot{H}^{-1}}\leq \frac{C_0}{\sqrt{1+(kt)^2}}\left[\|f^{in}_{k}\|_{u'}+\|\de_y f^{in}_{k}\|_{u'}\right],\qquad  \forall t\geq 0.
\end{align}
All the constants can be computed explicitly.
\end{theorem}

\begin{proof}[Proof of Theorem \ref{thm:mainFour}]
We will restrict ourselves to the proof of \eqref{eq:enMIXmonoFour}. The proof of \eqref{eq:engenFour2} is analogous, and we will
only highlight the differences whenever they arise. Again, dependences on $k$ of various functions are forgotten during the proof.
Set 
\begin{align}\label{eq:delta0}
\delta_0:=\frac{1}{4\times 3504 \UU^{6}  }.
\end{align}
From \eqref{eq:ODE5} and the above \eqref{eq:ODE11J} we deduce in particular that
\begin{align}\label{eq:ODEfinal}
\ddt \left[\Phi +\delta_0\J \right]+2\eps_0\nu^{1/3}k^{2/3}\left[\Phi +\delta_0\J \right]\leq 0.
\end{align}
By combining \eqref{eq:energy}, \eqref{eq:energyJ}, \eqref{eq:ERRL2J} and \eqref{eq:delta0}, we find that
\begin{align}
\ddt\left[\| f\|^2+\delta_0\|J f\|^2\right]+\nu\left[2\| \de_y f\|^2+\delta_0\| \nabla Jf\|^2\right]\leq \frac{\nu}{6\UU^2}\|Jf\|^2.
\end{align}
Since
\begin{align}
\|Jf\|^2\leq 2\|\de_yf\|^2 +2k^2t^2 \UU^2 \|f\|^2,
\end{align}
we find that
\begin{align}\label{eq:En4}
\ddt\left[\| f\|^2+\delta_0\|J f\|^2\right]+\nu\left[\| \de_y f\|^2+\delta_0\| \de_y Jf\|^2\right]\leq \frac13\nu k^2t^2  \|f\|^2.
\end{align}
In a similar manner, \eqref{eq:gamma}, \eqref{eq:gammaJ}, \eqref{eq:ERRgammaJ} and \eqref{eq:delta0} imply that
\begin{align}
\ddt\left[\| u'\de_x  f\|^2+\delta_0\|u'\de_x Jf\|^2\right]\leq7\UU^2\nu\left[\| u'\de_x  f\|^2+\delta_0\|u'\de_x Jf\|^2\right].
\end{align}
In particular, dividing everything by $k^2$, an application of the Gronwall lemma gives
\begin{align}\label{eq:ueff}
\| u'  f(t)\|^2+\delta_0\|u' Jf(t)\|^2\leq\e^{7\UU^2\nu t}\left[\| u'  f^{in}\|^2+\delta_0\|u' \de_yf^{in}\|^2\right], \qquad \forall t\geq 0.
\end{align}
Define 
\begin{align}
T_{\nu,k}=\frac{1}{\nu^{1/3}k^{2/3}}.
\end{align}
We divide the proof in two cases.
\subsection*{Case $t\geq T_{\nu,k}$}
We integrate \eqref{eq:En4} on $(0,T_{\nu,k})$ and use that $Jf(0)=\de_yf^{in}$ and  that
$\|f(t)\|\leq \| f^{in}\|$ to arrive at
\begin{align}
\nu\int_0^{T_{\nu,k}}\left[\| \de_y f(t)\|^2+\delta_0\| \de_y Jf(t)\|^2\right]\dd t 
&\leq \frac13\nu k^2\int_0^{T_{\nu,k}}t^2  \|f(t)\|^2\dd t+\| f^{in}\|^2+\delta_0\|\de_yf^{in}\|^2\notag\\
&\leq \nu k^2T_{\nu,k}^3\| f^{in}\|^2+\| f^{in}\|^2+\delta_0\|\de_yf^{in}\|^2\notag\\
&\leq 2\left[\| f^{in}\|^2+\delta_0\|\de_yf^{in}\|^2\right].
\end{align}
Thus, by the mean value theorem there exists $t_\star\in (0,T_{\nu,k})$ such that
\begin{align}
\nu{T_{\nu,k}}\left[\| \de_y f(t_\star)\|^2+\delta_0\| \de_y Jf(t_\star)\|^2\right]\leq 2\left[\| f^{in}\|^2+\delta_0\|\de_yf^{in}\|^2\right].
\end{align}
Rewriting and keeping in mind \eqref{eq:alphabeta} and that  $\alpha_0\leq 1$, it follows that
\begin{align}\label{eq:awes1}
\alpha\left[\| \nabla f(t_\star)\|^2+\delta_0\| \nabla Jf(t_\star)\|^2\right]\leq 2\left[\| f^{in}\|^2+\delta_0\|\de_yf^{in}\|^2\right].
\end{align}
On the other hand, by integrating \eqref{eq:En4} on $(0,t_\star)$ we have
\begin{align}\label{eq:awes2}
\| f(t_\star)\|^2+\delta_0\|J f(t_\star)\|^2
&\leq \frac13 \nu k^2\int_0^{t_\star}t^2  \|f(t)\|^2\dd t+\| f^{in}\|^2+\delta_0\|\de_y f^{in}\|^2\notag\\
&\leq \nu k^2T_{\nu,k}^3\| f^{in}\|^2+\| f^{in}\|^2+\delta_0\|\de_yf^{in}\|^2\notag\\
&\leq 2\left[\| f^{in}\|^2+\delta_0\|\de_yf^{in}\|^2\right].
\end{align}
Moreover, by \eqref{eq:ueff} and the restriction \eqref{eq:nukrestr}, we also deduce that 
\begin{align}\label{eq:ueff2}
\| u'  f(t_\star)\|^2+\delta_0\|u' Jf(t_\star)\|^2\leq\e^{7\UU^2\nu T_{\nu,k}}\left[\| u'  f^{in}\|^2+\delta_0\|u' \de_yf^{in}\|^2\right]
\leq2\left[\| u'  f^{in}\|^2+\delta_0\|u' \de_yf^{in}\|^2\right].
\end{align}
Now, using that from \eqref{eq:ODEfinal} we know that $\Phi+\delta_0\J$ is decreasing, we have from \eqref{eq:equiv}, \eqref{eq:equivJ}  and
\eqref{eq:awes1} and \eqref{eq:awes2} and \eqref{eq:ueff2} (and the fact that $\gamma_0\leq1$) that
\begin{align}\label{eq:awes3}
\Phi(T_{\nu,k})+\delta_0\J(T_{\nu,k})
&\leq \Phi(t_\star)+\delta_0\J(t_\star)\notag\\
&\leq \frac14\big[2(\| f(t_\star)\|^2+\delta_0\| Jf(t_\star)\|^2)+3\alpha(\|\de_{y} f(t_\star)\|^2+\delta_0\|\de_{y} Jf(t_\star)\|^2)\notag\\
&\qquad+3\gamma (\| u'\de_x  f(t_\star)\|^2+ \delta_0\| u'\de_x J f(t_\star)\|^2)\big]\notag\\
&\leq \frac14\big[2(\| f(t_\star)\|^2+\delta_0\| Jf(t_\star)\|^2)+3\alpha(\|\de_{y} f(t_\star)\|^2+\delta_0\|\de_{y} Jf(t_\star)\|^2)\notag\\
&\qquad+3\gamma_0 (\| u'  f(t_\star)\|^2+ \delta_0\| u' J f(t_\star)\|^2)\big]\notag\\
&\leq 2\left[\| f^{in}\|^2+\delta_0\| \de_yf^{in}\|^2+\|u' f^{in}\|^2+\delta_0\|u' \de_yf^{in}\|^2\right]\notag\\
&\leq 2\left[\| f^{in}\|_{u'}^2+\delta_0\| \de_yf^{in}\|_{u'}^2\right],
\end{align}
where the norm $\|\cdot\|_{u'}$ is defined in \eqref{eq:Xnorm}.
Thus, for $t\geq T_{\nu,k}$ we have from \eqref{eq:equiv}, \eqref{eq:equivJ}, \eqref{eq:ODEfinal} and \eqref{eq:awes3} that
\begin{align}\label{eq:expfinal2}
\frac{\delta_0\gamma_0}{4}\left[\| f(t)\|_{u'}^2+\| Jf(t)\|_{u'}^2\right] 
&\leq \Phi(t) +\delta_0\J(t)\notag \\
&\leq\e^{-2\eps_0\nu^{1/3}k^{2/3} (t-T_{\nu,k})}\left[\Phi(T_{\nu,k}) +\delta_0\J(T_{\nu,k}) \right]\notag\\
&\leq 2 \e^{2\eps_0}\e^{-2\eps_0\nu^{1/3}k^{2/3}t}\left[\| f^{in}\|_{u'}^2+\delta_0\| \de_yf^{in}\|_{u'}^2\right].
\end{align}
Thus since $\delta_0\leq 1$ and $\eps_0\ll1$,
\begin{align}\label{eq:expfinal3}
\|f(t)\|_{u'}^2+\| Jf(t)\|_{u'}^2\leq \frac{20}{\delta_0\gamma_0}\e^{-2\eps_0\nu^{1/3}k^{2/3}t}\left[\| f^{in}\|_{u'}^2+\| \de_yf^{in}\|_{u'}^2\right],\qquad \forall t\geq \frac{1}{\nu^{1/3}k^{2/3}},
\end{align}
as we wanted.

\subsection*{Case $t\in[0, T_{\nu,k})$}
Notice that we are done if we can prove that
\begin{align}
\| f(t)\|_{u'}^2+\| Jf(t)\|_{u'}^2\leq C\left[\| f^{in}\|_{u'}^2+\|\de_yf^{in}\|_{u'}^2\right], \qquad \forall t\leq \frac{1}{\nu^{1/3}k^{2/3}},
\end{align}
for some $C>0$.
From \eqref{eq:En4} and the Gronwall lemma, it follows that
\begin{align}
\| f(t)\|^2+\delta_0\| Jf(t)\|^2 \leq \e^{\nu k^2 t^3}\left[\| f^{in}\|^2+\delta_0\|\de_yf^{in}\|^2\right]\leq \e\left[\| f^{in}\|^2+\delta_0\|\de_yf^{in}\|^2\right].
\end{align}
Analogously, using \eqref{eq:ueff} (this is essentially \eqref{eq:ueff2}), we also have
\begin{align}\label{eq:ueff3}
\| u'  f(t)\|^2+\delta_0\|u' Jf(t)\|^2\leq\e\left[\| u'  f^{in}\|^2+\delta_0\|u' \de_yf^{in}\|^2\right].
\end{align}
Hence,
\begin{align}\label{eq:expfinal4}
\| f(t)\|_{u'}^2+\| Jf(t)\|_{u'}^2\leq \frac{\e}{\delta_0}\left[\| f^{in}\|_{u'}^2+\|\de_yf^{in}\|_{u'}^2\right], \qquad \forall t\leq \frac{1}{\nu^{1/3}k^{2/3}},
\end{align}
as we were trying to prove.
Putting together \eqref{eq:expfinal2} and \eqref{eq:expfinal4} we have
\begin{align}\label{eq:expfinal5}
\| f(t)\|_{u'}^2+\| Jf(t)\|_{u'}^2\leq \frac{20}{\delta_0\gamma_0}\e^{-2\eps_0\nu^{1/3}k^{2/3}t}\left[\| f^{in}\|_{u'}^2+\|\de_yf^{in}\|_{u'}^2\right],\qquad \forall t\geq0.
\end{align}
Having \eqref{eq:expfinal5} at our disposal, we only need to invoke Lemma \ref{lem:boundbelow}, so that
\eqref{eq:enMIXmonoFour} follows immediately. The proof of \eqref{eq:engenFour2} is exactly the same, just not considering the parts involving
the functional $\J$ in the above argument. The proof is therefore concluded. 
\end{proof}

\subsection*{Acknowledgements}
The author is indebted with T.M. Elgindi and K. Widmayer for endless and inspiring discussions, especially concerning 
the use of the vector field method in these problems, and with G. Iyer, A. Mazzuccato and C. Nobili for clarifications regarding the Batchelor scale.

\bibliographystyle{abbrv}
\bibliography{biblio}

\end{document}